\documentclass[11pt]{article}
\usepackage[utf8]{inputenc}
\usepackage[english]{babel}
\usepackage{amsfonts}
\usepackage[width=17cm, left=1cm]{geometry}
\usepackage{amssymb}
\usepackage{amsmath}
\usepackage{amsthm}

\newtheorem{theorem}{Theorem}

\newtheorem*{conjecture}{Conjecture}

\title{Central Measures of Continuous Graded Graphs:\\ the Case of Distinct Frequencies}

\begin{document}

\author{{A.~M.~Vershik \thanks{St.~Petersburg Department
of V. A.~Steklov Mathematical Institute of Russian Academy
of Sciences, St.~Petersburg State University, Institute for Information
Transmission Problems of Russian Academy of Sciences (Kharkevich Institute); supported by the RSF grant project no. 21-11-00152.}},
 {F.~V.~Petrov \thanks{St.~Petersburg Department
of V. A.~Steklov Mathematical Institute of Russian Academy
of Sciences, St.~Petersburg State University; supported by the RSF grant project no. 21-11-00152.}}}
\maketitle

%\date{\today}

\begin{abstract}
We define a class of continuous graded graphs similar to the graph of Gelfand--Tsetlin patterns, and describe the set of all ergodic central measures of discrete type on the path spaces of such graphs. The main observation is that an ergodic central measure on
a subgraph of a Pascal-type  graph can often be obtained as the restriction
of the standard Bernoulli measure to the path space of the subgraph. This observation
dramatically changes the approach to finding central measures also
on discrete graphs, such as the famous Young graph.

The simplest example of this type is given by the theorem on the weak limits of normalized Lebesgue measures on simplices; these are the so-called
Cesàro measures, which are concentrated on the sequences
with prescribed Cesàro limits (this limit parametrizes the corresponding measure).

More complicated examples are the
graphs of continuous Young diagrams with fixed number of rows
and the graphs of spectra of infinite Hermitian matrices of finite rank. We prove existence and uniqueness theorems for ergodic central measures and describe their structure.
In particular, our results
1) give a new spectral description of the so-called infinite-dimensional
Wishart measures~\cite{W}~--- ergodic unitarily invariant measures of discrete type
on the set of infinite Hermitian matrices;
2) describe the structure of continuous analogs
of measures on discrete graded graphs.

New problems and connections which appear are to be considered
in new publications.
\end{abstract}

%\textbf{ДАЛЕЕ СЛЕДУЕТ ОБНОВЛЕННЫЙ КУСОК ТЕКСТА}

 %\section{ОСНОВНЫЕ ОПРЕДЕЛЕНИЯ И ПОСТАНОВКА ЗАДАЧИ}

%\subsection{КОНТИНУАЛЬНЫЕ ГРАДУИРОВАННЫЕ ГРАФЫ, ИНЦИДЕНТНОСТИ, КОПЕРЕХОДНЫЕ ВЕРОЯТНОСТИ, ЦЕНТРАЛЬНЫЕ МЕРЫ. ОСНОВНАЯ ПРОБЛЕМА, ЧАСТОТЫ, ЭРГОДИЧЕСКЙ МЕТОД}.

% написать про общего ГЦ с линейно упорядоченным множеством

\section{Introduction}

Locally finite $\mathbb{N}$-graded graphs,
often called Bratteli diagrams,
proved to be a useful instrument
both for algebraic and stochastic problems. It is important to stress
that the geometry of such graphs
is related to the theory of discrete
topological Markov chains, although
the usual questions about Markov chains differ from the questions of the theory of Bratteli diagrams. Combining these two theories
may yield interesting results.
It is well known in the theory of Markov chains  that it is  fruitful to consider
not finitely or countably many
states, but continuously many states.
This approach is well developed in probability
theory, although some generalizations
of notions and results are nontrivial.
But in the theory of graded
graphs, it is not yet developed.
In this paper we want to
consider first examples of results
concerning so-called continuous graded
graphs. We recall that the theory
of continuous graphs is a new and active branch \cite{L},
but the theory of graded graphs has its own problems which are somehow
different from the usual graph theory
problems.

A basic notion for a graph
is the path space, possibly structured.
For a graded graph, we consider
 ``graded'' paths which begin at the
first level and go down.
In this paper we consider graphs in which both the levels and the sets of  edges between consecutive
levels correspond to
convex cones in Euclidean spaces. Such graphs are said to be of \emph{Gelfand--Tsetlin type}.

The simplest
such graph has the half-line $\mathbb{R}_+$ as the vertex set at each level,
and there is an edge from $x$ to $y$ iff $x\leqslant y$. This makes it possible
to identify a path in this graph with a series with nonnegative terms.
We call it the \emph{Cesàro} graph. We are especially interested in the following
two generalizations of the Cesàro graph. The first one is the
continuous Gelfand--Tsetlin graph:  the vertices at level $n$ are
the increasing sequences $x_1\leqslant x_2\leqslant\ldots\leqslant x_n$
of real numbers, and there is an edge between two sequences
$x_1\leqslant x_2\leqslant\ldots\leqslant x_n$ and $y_1\leqslant y_2\leqslant\ldots\leqslant y_{n+1}$ iff they interlace:
$$y_1\leqslant x_1\leqslant y_2\leqslant x_2\leqslant\ldots\leqslant x_n\leqslant y_{n+1}.$$
This graph has the subgraph consisting of nonnegative sequences (``positive Gelfand--Tsetlin graph''). An important intimately related graph is the so-called \emph{rank $d$
Gelfand--Tsetlin graph}, in which the vertices at level $n$ are the
increasing sequences
$x_1\leqslant x_2\leqslant\ldots\leqslant x_d$ of $d$ nonnegative numbers,
and an edge corresponds to the interlacing condition
$$x_1\leqslant y_1\leqslant x_2\leqslant y_2\leqslant\ldots\leqslant x_d\leqslant y_{d}.$$
This graph may be regarded as a subgraph of the positive Gelfand--Tsetlin graph by considering only the levels  $\geqslant d$ and sequences of the form $(0,0,\ldots,0,x_1,\ldots,x_d)$
(with $n-d$ initial zeros). These graphs correspond to the spectra of corners
of Hermitian matrices (arbitrary, or nonnegative definite, or nonnegative definite of rank at most $d$).

For an infinite nonnegative definite
Hermitian matrix $A$ of rank at most $d$,
we may define $A_n$ for $n\geqslant d$ as the principal
$n\times n$ minor of $A$, and let
$x_1\leqslant \ldots \leqslant x_d$
be the $d$ largest eigenvalues of $A_n$
(all the other eigenvalues are equal to 0).
Thus we get a path in
the Gelfand--Tsetlin graph of rank $d$.
The pushforward of measures
on nonnegative definite
infinite Hermitian matrices of
rank $d$ are, therefore, measures
on the path space of the Gelfand--Tsetlin graph of rank $d$, and $U(\infty)$-invariant
measures on matrices correspond to
central measures on the path space.
For matrices, we get the so-called
\emph{Wishart measures}. The Wishart measure of rank $d$
with frequencies $\lambda_1,\ldots,\lambda_d$ is defined as
follows: take $d$ i.i.d.\ standard complex Gaussian vectors
$\xi_1,\xi_2,\ldots,\xi_d\in \mathbb{C}^n$ and consider the matrix
$\sum \lambda_i \xi_i\cdot \xi_i^*$. Without rank restrictions, the $U(\infty)$-invariant
measures include, besides Wishart measures, which now may have infinitely many frequencies, the so-called Gaussian unitary ensemble (GUE), corresponding to Gaussian matrix entries (real on the main diagonal and complex otherwise; the entries on the  diagonal and above the diagonal are independent).

Possibly the
first paper in which the Gaussian unitary ensemble
is considered
in the framework of
central measures on graded graphs is \cite{Bar}.
The complete list of invariant
measures for infinite Hermitian  matrices
(and other similar series)
is known for a while and was obtained by different methods by Pickrell
\cite{Pic} and Olshanski and Vershik
 \cite{VO96}
 (the latter paper also uses the ergodic method but in a
 very different way from what we mean here).

 This problem was considered in an old
 paper  \cite{V74} of the first author
 about the ergodic method, but the list of Wishart measures
 was missed there.
 The invariant measures on positive definite Hermitian
 matrices are easier to find, this problem was
 addressed in \cite{DS,Aus}.

 Gaussian measures ($GUE, GOE$) on matrices
 are especially popular. But we study these measures in terms of the graph of  spectra of matrices rather than the graph of matrices themselves (for the $GUE$, this is done in \cite{Bar}). The spectral approach is important not only in itself, but also because it allows us to discover more subtle properties of invariant measures.  Moreover, this point of view is also possible for other fields (see, for example, \cite{Meta}): for self-adjoint and
 quaternionic matrices. For more
 on the relation between
 graphs of matrices and graphs of spectra, see
 \cite{VP}.

Another important generalization is the graph of increasing sequences $x_1\leqslant \ldots
\leqslant x_d$ of nonnegative real numbers in which there is an edge between two  sequences iff
$x_i\leqslant y_i$ for all $i=1,\ldots,d$. In other words, vertices
are continuous Young diagrams with at most $d$ rows, and an edge corresponds
to one diagram being contained in the other.

A probabilistic Markov chain
is determined by transition
probabilities, but in the theory of Bratteli diagrams, like in many other problems in
mathematics and physics \cite{V21}, the basic notion is that of \emph{cotransition probability}. Cotransition, or cocycle,
is the main additional structure on a graded graph, and here we chose the cotransition probability measures to be the \emph{normalized
Lebesgue measures on the phase spaces of
cotransitions}. These cotransitions
arise in the theory of
invariant measures on Hermitian matrices \cite{Bar}. The Lebesgue cocycle
is a continuous analog of the uniform cocycle for discrete graded graphs.
A measure on the path space of a graded graph with Lebesgue cocycle
is called a \emph{central measure}.

Next, a remarkable problem arises of
describing all central measures for a given graph.
We discuss it for specific cases. For many countable graphs
(the Pascal, Young, Schur graphs,
dynamical graphs on groups and random
walks), this problem has been considered
in the recent decades.

A natural method for solving this problem
is the ergodic method \cite{V74,V21},
see the details below. Applying it
to find central
measures on continuous
graphs is especially spectacular
because of the geometric nature of the cocycle.

In many cases, the list of ergodic central measures is conjecturally known (or even already found by other methods). For example, in our case, each such measure is given by a set of finite or countable frequencies of natural objects. The problem is
 to explain why the frequencies uniquely determine an ergodic central measure, why they are determined by it, and, most importantly,  how the frequency description of the measure is related to the corresponding product measure.
 Up to now, there was no such description not only for continuous graphs, but even for the Young graph. This is what is done below for the chosen class of graded graphs and for some discrete graphs.

Now we outline the contents of the paper. We begin with the fundamental example of the continuous Cesàro graph. The cocycle here is the set of normalized
Lebesgue measures on simplices $\{x_i\geqslant 0,\sum x_i=S\}$ of
growing dimension with the sum $S$ growing linearly with the dimension.
It is easy to understand that a central measure on this graph is the weak limit of Lebesgue measures on simplices, but it is more difficult to prove that
there exists a unique central measure for each
growth rate, i.e., for each frequency.
The structure of the measure is clear: it is a measure on the set of Cesàro
sequences, which explains the name.
Cesàro measures play here the same role as Bernoulli measures for discrete graphs.

For graphs more general than the Cesàro graph, that is, for the finite rank
Gelfand--Tsetlin graphs and for the graph of continuous Young diagrams,
\emph{the central measures with distinct
frequencies are the restrictions of direct products of
Cesàro measures to the corresponding sets of
paths}. In the case where some frequencies are equal, some regularization of the weak limits is needed. But, quite unexpectedly, it turns out that exactly the same, and even simpler, answer can be obtained for the problem of describing the central measures of discrete type for many countable graphs, e.g., for the Young graph.

Note that in \cite{VK81} all central measures for the Young
graph were realized by means of a generalized RSK algorithm
as homomorphic images of Bernoulli measures,
and in the subsequent paper \cite{RomSni}
it was proved that this homomorphism is an isomorphism.
However, this approach, while important, does not prove the completeness of the list of central measures. It turns out that the completeness of the list in this and other cases follows from an entirely different relation between Bernoulli measures and central measures: {\it the latter are not only factors of Bernoulli measures, but also restrictions to their proper subsets}. To the authors' knowledge, for some reason this fact has been so far overlooked by
mathematicians.

The class of continuous graphs under consideration
is wide enough. It is convenient to regard infinite
paths in these graphs as \emph{positive vector series}.
The dimension of the vector being equal to $1$
corresponds to the case of the Cesàro graph.

We emphasize that the main conclusion from our considerations is that the ergodic central measures for a rather wide class of continuous graphs are determined by their frequencies,
i.e., growth rates of specific coordinates.
More precisely, a central measure is a Markov
measure on positive vector series, and the behavior of these random series reflects the asymptotic behavior of frequencies.

 It is important to stress that for the Gaussian unitary ensemble (for the Plancherel measure in the Young graph case), the picture is quite different, since all the limiting frequencies are zero; $GUE$ (as well as the Plancherel measure) owes its existence to a very nontrivial interaction between the positive and negative parts of the
 spectra, each of which has a sublinear (rather than linear) asymptotics.

A remarkable parallel between these examples will be discussed in another paper. We are grateful
to a referee who drew our attention
to the interesting paper \cite{Buf}
where the connection between these two measures
is established via the values of moments.
But we mean a different type of parallelism.

We emphasize that the calculations involved in the ergodic method are sometimes quite nontrivial. In this paper we observe the very fact that central measures are determined by frequencies and find them explicitly in the simplest case of the so-called Cesàro graph. Further calculations are the subject of a forthcoming publication.

\section{Setup}

The vertex set of an $\mathbb{N}$-graded graph is
$$\Gamma=\bigsqcup_{n=0}^{\infty}\Gamma_n,$$
where $\Gamma_n$ is the set of vertices of level $n$ and every edge joins two vertices of consecutive levels.
Here we assume that the following two conditions are satisfied.

\smallskip
1. Euclideness. All levels $\Gamma_n$, $n=0,1,2 \dots$, are identified
with closed finite-dimensional convex polyhedral cones $\Gamma_n\subset {\mathbb R}^{d_n}$
with nonempty interior (but the levels in a graph
are disjoint, so, rigorously speaking, we must write something like $\Gamma_n\subset
{\Bbb R}^{d_n}\times \{n\}$). The dimensions $d_n$ may or may not depend on $n$.

\smallskip
2. Convex incidence. Two vertices $x \in \Gamma_n$, $y \in \Gamma_{n+1}$
of consecutive levels are joined by an edge
iff the pair ${x,y}$
belongs to a certain closed convex polyhedral cone
$$D_n \subset \Gamma_n \times \Gamma_{n+1}, \qquad n=0,1,\dots.$$

We require  the natural
projections $D_n\to \Gamma_n$ and $D_n\to \Gamma_{n+1}$ to be surjective; in
other words, every vertex in $\Gamma_0$ has a neighbor in $\Gamma_1$,
and every vertex in $\Gamma_n$, $n>0$, has neighbors in both  $\Gamma_{n-1}$
and $\Gamma_{n+1}$.

An infinite path is defined as a sequence of vertices
$(v_0,v_1,\ldots)$, $v_i\in V_i$, such that $(v_i,v_{i+1})\in D_i$
for all $i=0,1,\ldots$. The path space is the set
of all infinite paths endowed with a natural (weak) topology.

\smallskip
3. Nondegeneracy. For all interior points
$x \in \Gamma_n$, $z \in \Gamma_{n+1}$, the section
$D^{+}_n(x)=\{y\in \Gamma_{n+1}: (x,y)\in D_n\}$ (the set of second
endpoints of the edges from $x$ to $\Gamma_{n+1}$)
is a nondegenerate convex polyhedral set that does not contain
lines;
and the section  $D^{-}_n(z)=\{y\in \Gamma_n: (y,z) \in D_n\}$ (the set of starting
points of the edges from $\Gamma_n$ to $z$) is a nondegenerate
convex polytope. (The nondegeneracy here means the maximum dimension).
This property can be weakened, but it determines an
important special class of continuous graphs.
\smallskip

A certain class of continuous graphs is already defined.
We proceed to specify the class of graphs under consideration by imposing new conditions on the graph structure.

\smallskip
4. Homogeneity. This condition concerns
the case where all levels
are isomorphic: a~graph is said to be \emph{homogeneous}
if the incidence cones $D_n$ coincide for all $n$.

\smallskip
Graphs satisfying the above properties
1--3 are called {\it continuous graphs of Gelfand--Tsetlin type}, and if all levels of such a graph
coincide and property~4 holds, we call
it a {\it homogeneous  continuous graph of Gelfand--Tsetlin type}.

\section{Cocycles}

 We proceed to the definition of cotransition probabilities
 and cocycle for graded graphs. They are introduced in the same
 (and even simpler) manner as for discrete graded graphs.

 For every interior vertex $x\in \Gamma_n$, $n>0$, the paths from
 $\Gamma_0$ to $x$ form a compact set $\mathcal{P}(x)$
 of dimension $d_0+d_1+\ldots+d_{n-1}$, which is
 stratified in a natural way into sets $\mathcal{P}(y)$ for $y\in D_n^{-}(x)\subset \Gamma_{n-1}$.
 Assume that we have already fixed a probability measure $\mu_y$ on each
 $\mathcal{P}(y)$ for $y\in D_n^{-}(x)\subset \Gamma_{n-1}$. Then, if we fix a
 probability measure on $D_n^{-}(x)$ (so-called cotransition probability measure),
 it defines a probability measure $\mu_x$ on $\mathcal{P}(x)$. So, by an obvious induction,
 a system of
 cotransition measures defines a measure on the paths to $x$ for every vertex $x$
 of $\Gamma$. More generally, it defines a~probability measure
 on the paths from $y$ to $x$
 for any two vertices $x$ and $y$ such that there exists a~path from $y$ to $x$.
 Hereafter, we assume that the measures are absolutely continuous with
 respect to the Lebesgue measures in the corresponding Euclidean spaces. Then,
 for tail-equivalent paths $P_1,P_2$
 (that is, infinite paths that coincide from some point on)
 we can define a \emph{cocycle} $c(P_1,P_2)$ as follows. Take
 a  large number $n$ such that $P_1,P_2$ coincide after  level $n$
 and cut the paths $P_1,P_2$ at this level. We get two paths $P_1(n), P_2(n)$
 belonging to the same set $\mathcal{P}(x)$ for certain $x\in \Gamma_n$.
 The ratio of the densities of $\mu_x$ at $P_2(n)$ and $P_1(n)$ is denoted by
 $c(P_1,P_2)$. It is straightforward that this number does not depend on
 the choice of  $n$.

 In the discrete case, it is natural to consider the uniform distributions on the sets
 $\mathcal{P}(x)$, which correspond to  the \emph{central cocycle}
 $c(P_1,P_2)\equiv 1$ (instead of considering densities, here we may simply divide probabilities).
 Borel
 measures on infinite paths with such a cocycle are called \emph{central measures}. As follows from the definition, they are Markov measures (with time corresponding
 to the grading),  but not vice versa.
 The set of ergodic (indecomposable) central measures is called the
 \emph{absolute} of the graph. Finding the absolute is the most important problem
 in the theory of graded graphs.

 For the class of continuous Gelfand--Tsetlin type graphs introduced above,
 we introduce a natural analog of the central cocycle, which we call the
 \emph{Lebesgue cocycle}. It corresponds to considering the normalized Lebesgue measure
 on each $\mathcal{P}(x)$, in other words, again $c\equiv 1$. Measures with such a~cocycle are again called central measures, and the main problem is again to describe all
 ergodic central measures, i.e., the \emph{absolute} of Gelfand--Tsetlin type graphs.

 More general cocycles are to be considered separately.

\section{Ergodic method}

The so-called ergodic method
for describing central and invariant measures  was suggested in \cite{V74}. It essentially follows from the individual ergodic theorem or the
martingale convergence theorem.
To state the method, we introduce the necessary notion of weak convergence of measures on
the path space.
 A sequence of measures on the path space of a graph is weakly convergent if for any cylinder set the values of the measures of this set converge.
In our situation, the ergodic method can be described as follows
(cf.~\cite{V74,V21}).

\begin{theorem}\label{tt}
For every ergodic central measure $\mu$ on the
path space of a Gelfand--Tsetlin type graph~$\Gamma$
there exists a path, i.e., a sequence of vertices
$\{x_n\}_n$, $x_n \in D_n^-(x_{n+1})$,  $n=1,2,\dots$,
for which the sequence of the Lebesgue measures on the sets $\mathcal{P}(x_n)$
weakly converges to $\mu$.
\end{theorem}

Although the
continuous setting is formally different from
what has been considered earlier,
the usual proof holds verbatim.

This theorem reduces the problem of describing the absolute to concrete calculations,
which may be laborious in practice. On the other hand,
the  conditions of weak convergence in concrete cases may already be restrictive enough and
allow one to describe the absolute explicitly.

The ergodic method was used in \cite{VK81} to describe
the characters of the infinite symmetric group $S_\infty$, in \cite{VO96} to describe
the invariant measures on infinite Hermitian matrices, and in many other
problems about invariant or central measures.

\section{Frequencies}

Here we restate the main problem of describing the absolute of
a Gelfand--Tsetlin type graph in concrete coordinate terms.

Assume that each level has dimension $d$ and
coincides with the cone $$\mathcal{C}:=\{(x_1,\ldots,x_d)\in \mathbb{R}^d\colon 0\leqslant x_1\leqslant x_2\leqslant \ldots \leqslant x_d\}.$$

A convex cone $D_n$ (which is also the same for all
$n$) is defined by several linear inequalities
satisfied by two vertices $x=(x_1,\ldots,x_n)\in \Gamma_n$ and $y=(y_1,\ldots,y_n)\in \Gamma_{n+1}$.

Our two main examples are as follows.

\smallskip
(i) The  {\it rank $d$ Gelfand--Tsetlin graph} is defined by the
interlacing inequalities $$
D_n=\mathcal{D}_{GT}:=\{x,y\colon 0\leqslant x_1\leqslant y_1\leqslant x_2\leqslant y_2\leqslant \ldots \leqslant x_d\leqslant y_d\}.
$$
This graph arises in  random matrix theory.
Namely, consider an $n\times n$
Hermitian matrix
with eigenvalues $0$ (of multiplicity $n-d$),
$y_1,\ldots,y_d$, where $0\leqslant y_1\leqslant \ldots\leqslant y_d$.
Its principal minor of order $n-1$ has
eigenvalues $0$ (of multiplicity $n-d-1$) and
$x_1,\ldots,x_d$, where
$
0\leqslant x_1\leqslant y_1\leqslant x_2\leqslant y_2\leqslant \ldots \leqslant x_d\leqslant y_d$.
Moreover, a unitarily invariant
(in the natural sense) measure on the orbit of
matrices with fixed eigenvalues and fixed
eigenvalues of the minor of order $k$, $d\leqslant k\leqslant n$, is
the Lebesgue measure on the corresponding polytope.
This is proved by Baryshnikov \cite{Bar} in greater generality.

So, along with the Gelfand--Tsetlin graph, we may consider a ``covering''
graph of matrices: at level $k$, the vertices are the $k\times k$ Hermitian matrices  of
rank at most $d$, and an edge joins a $k\times k$
matrix with its principal $(k-1)\times (k-1)$ minor.

The ordinary Gelfand--Tsetlin graph is defined similarly, but without
the rank restriction.

Note that the definition of our graphs uses only inequalities between elements, and thus it can be automatically considered over an arbitrary linearly ordered set (for example, over a segment instead of the half-line) and even a poset.

In the discrete case of the linearly ordered
set $\mathbb{Z}$, it is related to the
representation theory of unitary groups, see, for example,
\cite{BO}.
%написать подробнее про граф эрмитовых матриц, написать про бету и в вещественном случае

\smallskip
(ii) {\it The Young jumps graph}. Here $D_n$ is the following cone:
$$
\mathcal{D}_{Y}:=\{x,y\in \mathcal{C}\colon x_1\leqslant y_1, x_2\leqslant y_2,
 \ldots ,x_d\leqslant y_d\}.
$$
If we identify a point $(x_1,\ldots,x_d)$
with the Young diagram with $d$
rows of (not necessarily integer) lengths
$x_1,\ldots,x_d$, then the condition $(x,y)\in \mathcal{D}_{Y}$
means that one diagram is
contained in the other. In the ordinary discrete
Young graph, an edge corresponds to the condition
``one diagram is contained in the other and differs
from it
by exactly one box.'' Here we may
remove (or add, if we consider
moving along a forward path)
arbitrarily many ``boxes.''

\medskip
Let $x(n)=(x_1(n),x_2(n),\ldots,x_d(n))\in \Gamma_n$, $n=1,2,\ldots$, be a path defining a central probability measure $\mu$
on a homogeneous Gelfand--Tsetlin type graph $\Gamma$
according to Theorem \ref{tt}.
If $\lim \frac{x_i(n)}n=\lambda_i\in [0,\infty]$, we say that the $\lambda_i$'s
are the \emph{frequencies} of the corresponding measure. We always may
pass to a subsequence of  $\{x(n)\}$ for which the frequencies
exist. The importance of this notion is due to the fact that
in many cases (and we conjecture that for homogeneous Gelfand--Tsetlin type graphs, this is always the case) the frequencies uniquely
determine an ergodic central measure.

\section{Restricting and recovering}

Let $\Gamma$ be a graded graph with levels
$\Gamma_n$ and
$\tilde{D}_n\subset D_n$
be (measurable) subsets of the edge sets of $\Gamma$.
Denote by
$\tilde{\Gamma}$ the subgraph induced on
these subsets.

Let $\mu$
be a central
probability measure on the path space of
$\Gamma$. Assume that with positive
$\mu${\nobreakdash-}probability a random path is a path in
$\tilde{\Gamma}$, that is, it goes only along
the edges in $\sqcup \tilde{D}_n$.

Our crucial observation is the following
almost obvious theorem.

\begin{theorem}\label{l1}
The restriction of $\mu$ to the path space
of $\tilde{\Gamma}$ is also a central measure.
\end{theorem}

\begin{proof}
Fix a level $n_0$. Like every central measure, $\mu$
enjoys the Markov property, that is,
it is characterized by the distribution
at the  level $\Gamma_{n_0}$
and the
conditional measures on the tails after
this level. The initial segments before
$\Gamma_{n_0}$ are distributed (Lebesgue) uniformly
on the corresponding polytopes
$\mathcal{P}(x)$, $x\in \Gamma_{n_0}$,
and do not depend on the tails.
All this is preserved under restriction,
since the restriction of a Lebesgue measure is
again a Lebesgue measure.
\end{proof}

Note that the theorem does not
give a recipe for obtaining
specific formulas for
the restricted measure, and this task can be quite
nontrivial. For example,
consider the ordinary (discrete) Young graph consisting
of Young diagrams with at most $k$
rows. It may be viewed as the subgraph consisting of decreasing
integer sequences $n_1\geqslant
n_2\geqslant \ldots \geqslant n_k\geqslant 0$
in  the Pascal graph $\mathbb{Z}_{\geqslant 0}^k$.
Consider the Bernoulli measure with
frequencies $p_1>p_2>\ldots>p_k$
on $\mathbb{Z}_{\geqslant 0}^k$ (which means
that at each step we increase the
$i$th coordinate by 1 with probability $p_i$
and do not change the other coordinates). Its
restriction to the Young graph is the Thoma
measure with parameters $p_1,\ldots,p_k$:
the probability of a~diagram with
row lengths
$n_1\geqslant
n_2\geqslant \ldots \geqslant n_k\geqslant 0$
is the value of the corresponding
Schur function
at the point $(p_1,\ldots,p_k)$.

A counterpart of Theorem \ref{l1} states that
a finite measure $\nu$ on the path space of the
subgraph
$\tilde{\Gamma}$
can be extended to a measure
$\mu$ on the path space of the whole
graph $\Gamma$ such that the restriction of~$\mu$ to the path space of $\tilde{\Gamma}$
coincides with $\nu$ (but this
$\mu$ is not always finite, this is a
delicate issue).

Indeed, the centrality
property allows us to define $\mu$
on the images of the path space of
$\tilde{\Gamma}$ under transformations
that preserve central measures.
For example, we may consider transformations that fix the tail after level $n_0$
and change the initial segments accordingly
(preserving the Lebesgue measure).
This defines $\mu$ on the set of paths that are eventually in $\tilde{\Gamma}$. It is not
hard to see that this definition is consistent, and we
get a measure on the path space of $\Gamma$
supported on the paths that are eventually
in $\tilde{\Gamma}$.

If the resulting measure is finite, then the original measure
on the path space of $\tilde{\Gamma}$ can be obtained by restriction
(and normalization). For homogeneous Gelfand--Tsetlin type
graphs, as well as for the discrete Young graph, this corresponds
to the case of \emph{distinct frequencies}.

\section{One-dimensional case: Cesàro measures}

Here we deal with the $d=1$ case of the Gelfand--Tsetlin graph
(the Young jumps graph coincides with the
Gelfand--Tsetlin graph for $d=1$) and compute the
central measures of this graph.

A path in our graph is simply a sequence
$0\leqslant x_0\leqslant x_1\leqslant \ldots$
of nonnegative numbers, and the path space
$\mathcal{P}(a)$ for a vertex
$a\in \Gamma_n=[0,\infty)$  is the $n$-dimensional
simplex
$$\Delta_n(a):=\{0\leqslant x_0\leqslant x_1\leqslant \ldots\leqslant x_n=a\}.$$
 We must study
the distribution of finitely many
coordinates $x_0,\ldots,x_{k-1}$ (where $k$
is fixed and $n$ is large) with respect
to the normalized Lebesgue measure $\mu_a$
on this simplex.

Consider an ergodic central measure
on the one-dimensional
Gelfand--Tsetlin graph. By Theorem~\ref{tt},
it corresponds to a certain sequence of vertices
(even to a path, but it is more convenient to consider
a~sequence: the difference is that in a
sequence some levels may be skipped).
Passing to a subsequence, we may assume that
the vertices $a_i\in \mathbb{R}_+$ at levels $n_i$
satisfy the relation $\lim a_i/n_i=\lambda$, where
$\lambda\in [0,\infty]$.

Thus,  the following theorem implies that
every ergodic central measure
on the one-dimensional
Gelfand--Tsetlin graph is an exponential random
walk over the levels of the graph.

\begin{theorem}
Let $\lambda>0$ be constant
and $a=\lambda n+o(n)\in \Gamma_n$
be a vertex of level $n$ of the one-dimensional
Gelfand--Tsetlin graph. Then for every fixed
positive integer $k$, the measure induced
by~$\mu_a$
on the sequences $x_0<x_1<\ldots<x_{k-1}$
converges to the exponential random walk
with mean $\lambda$:\break $x_0,x_1-x_0,\ldots,x_{k-1}-x_{k-2}$ are  i.i.d.\ with distribution
$\operatorname{Exp}(\lambda)$.

If $a=o(n)$, then the corresponding distributions weakly converge
to the $\delta$-distribution at the zero sequence $(0,0,\ldots)$;
 if $a/n\to \infty$, then
the corresponding distributions weakly converge to the zero measure.
\end{theorem}

Because of the importance of this theorem
and its relation to generalizations,
we provide
three different proofs. We focus on the case $0<\lambda<\infty$,
the proofs for $\lambda=0$ and $\lambda=\infty$ are similar (and even
simpler).

\begin{proof}[Proof 1]
Denote $y_0=x_0$, $y_i=x_i-x_{i-1}$ for $i=1,\ldots,n-1$.
Then our simplex is defined as $\{y_i\geqslant 0,\sum y_i\leqslant a\}$.
Its volume equals $a^n/n!$.
If we restrict the distribution to the first $k$ coordinates
$(y_0,\ldots,y_{k-1})$, then
the density at the above point is proportional to the
volume of the corresponding section $\{y_k+\ldots+y_{n-1}\leqslant a-(y_0+\ldots+y_{k-1})\}$. Denoting $s=y_0+\ldots+y_{k-1}$,
we see that this volume equals $(a-s)^{n-k}/(n-k)!$. Thus, the density
equals
$$
\frac{(a-s)^{n-k}}{x^{n-k}\cdot a^k}\cdot{n!}{(n-k)!}=
(1-s/a)^{n-k}\cdot \frac{n(n-1)\ldots (n-k+1)}{a^k}
\sim e^{-s/\lambda}\lambda^{-k}=\prod_{i=1}^k \lambda^{-1}e^{-y_i/\lambda},
$$
which proves the claim.
\end{proof}

\begin{proof}[Proof 2]
It is more convenient to study the distribution of
$a-x_{n-1}$, $x_{n-1}-x_{n-2}$, $\ldots$,
$x_{n-k+1}-x_{n-k}$. Of course, this vector has
the same distribution as $x_0,x_1-x_0,\ldots,x_{k-1}-x_{k-2}$, as can be seen from the measure-preserving
automorphism $(x_0,\ldots,x_{n-1})\to
(a-x_{n-1},\ldots,a-x_0)$ of $\Delta_n(a)$.

Consider the following map from $\Delta_n(a)$
to the unit cube $[0,1]^n$, which maps
the normalized Lebesgue measure on the simplex to the Lebesgue measure on the cube:
$$
\Phi_n\colon (x_0,x_1,\ldots,x_{n-1})\to
\left(\frac{x_0}{x_1},\frac{x_1^2}{x_2^2},
\ldots,\frac{x_{n-2}^{n-1}}{x_{n-1}^{n-1}},
\frac{x_{n-1}^n}{a^n}\right)=:(t_0,\ldots,t_{n-1}).
$$
The Jacobi matrix of $\Phi_n$ is triangular,
which immediately implies that its Jacobian equals
$n!$. Thus, $\Phi_n$ indeed maps
the uniform distribution on $\Delta_n(x)$
to the uniform distribution on $[0,1]^n$.
Hence,
\begin{align*}
    a-x_{n-1}&=a\left(1-t_{n-1}^{1/n}
    \right)\sim -\frac{a}n
\log t_{n-1},\\
x_{n-1}-x_{n-2}&=x_{n-1}
\left(1-t_{n-2}^{1/(n-1)}\right)=
x\cdot t_{n-1}^{1/n}\cdot
\left(1-t_{n-2}^{1/(n-1)}\right)\sim
-\frac{x}n \log t_{n-2},\\
x_{n-2}-x_{n-3}&=x_{n-2}
\left(1-t_{n-3}^{1/(n-3)}\right)\sim
-\frac{a}n \log t_{n-3},\\
\ldots&\ldots \ldots,
\end{align*}
which yields the result, since $a/n$ goes to
$\lambda$ and $-\lambda \log T\in \operatorname{Exp}(\lambda)$
when $T\in \operatorname{Unif}(0,1)$. (In the above argument, $a\sim b$ means that $a/b$ converges to 1 in probability. We have used the fact that $t^{1/n}\sim 1$ for $t$ distributed uniformly on $(0,1)$.)
\end{proof}

\begin{proof}[Proof 3] This proof, based on the law of large
numbers, is ideologically close to the proof of the
Maxwell--Poincaré lemma, which states that the distribution
of coordinates of a Lebesgue-random point on
the $n$-dimensional sphere
of radius $\sqrt{n}$ is asymptotically standard normal.

We start with the distribution $\eta_n$ of a vector
$y\in (0,\infty)^n$
defined as the product of the $\operatorname{Exp}(\lambda)$
distributions
over all $n$ coordinates,
and prove that the distribution
of the first $k$ coordinates of this vector
is close to that of a random point
in our simplex $\Delta_n(a)$.
The density of $\eta_n$ at a point $y=(y_1,\ldots,y_n)$
equals $$p(y)=\prod_{i=1}^d \lambda^{-1}e^{-y_i/\lambda}=\lambda^{-n}
e^{-(y_1+\ldots+y_n)/\lambda}.$$
Thus, we see that $p(y)$ depends only on the sum $y_1+\ldots+y_n$ of the coordinates.
In other words, the random point
$y$ can be obtained by choosing the value of
$x=y_1+\ldots+y_n$ at random
(according to a certain distribution,
which is in fact a gamma distribution, but we do
not use this)  and then choosing a uniformly distributed random point in the simplex
$S_x:=\{y_i\geqslant 0,\sum y_i=x\}$.
Now fix a set $\Omega\subset{\mathbb{R}^k}$ that is
a product of intervals. The probability that $(y_1,\ldots,y_k)\in \Omega$ is equal to the expectation
against~$x$ of the probability
that $(y_1,\ldots,y_k)\in \Omega$
for a random point $(y_1,\ldots,y_n)\in S_x$.
Note that $(y_1,\ldots,y_k)\in \Omega$
if and only if
$\frac{\lambda n}x(y_1,\ldots,y_k)\in
\frac{\lambda n}x \Omega$, and the point
$\frac{\lambda n}x(y_1,\ldots,y_n)$
is uniformly distributed
in $S_{\lambda n}$. Now observe that,
by the law of large numbers, for any
$\varepsilon>0$ the probability that
$\frac{\lambda n}x\notin [1-\varepsilon,
1+\varepsilon]$ tends to 0.
When $\tau\in [1-\varepsilon,
1+\varepsilon]$, the sets $\tau \Omega$
vary a little. In particular, their
intersection contains a set $\Omega_{-}(\varepsilon)$
and their union is contained in a set
$\Omega_+(\varepsilon)$ such that the
$\eta_k$-measures of both sets are close to
each other. This yields the required
weak convergence.
\end{proof}

Thus, any ergodic measure on the one-dimensional Gelfand--Tsetlin graph
is concentrated on the series $x_0+(x_1-x_0)+(x_2-x_1)+\ldots$
that are Cesàro summable to a certain number $\lambda>0$.
That is why we suggest to call it a \emph{Cesàro measure}.

\section{Multi-dimensional case}

Here we explain how combining the one-dimensional
Cesàro case and Theorem \ref{l1} allows us to
handle the case of a finite rank Gelfand--Tsetlin
graph and the Young jumps graph.

Note that the rank $d$
Gelfand--Tsetlin graph is a subgraph
of $\mathbf{Ces}^d$,
where $\mathbf{Ces}$ is the Cesàro
graph. On $\mathbf{Ces}^d$
we can consider a Bernoulli measure that
is the product of Cesàro measures
with frequencies $p_1<p_2<\ldots<p_d$.
With positive probability, any
vertex $(x_1,\ldots,x_d)$
is an increasing sequence and
any two consecutive
vertices $(x_1,\ldots,x_d)$ and
$(y_1,\ldots,y_d)$ interlace
(by the law of large numbers, with probability 1 this happens eventually, and it can easily
be seen that with positive probability this
``eventually'' is from the very beginning).
Thus, by our recovering procedure we get a unique
ergodic central measure on the path
space of the Gelfand--Tsetlin graph with given distinct
frequencies.

A measure for which some
frequencies are equal can be obtained as a weak limit
of measures with distinct frequencies, but we
skip the details here. All the same
holds for the Young jumps graph: the ergodic central measures
are in one-to-one correspondence with arrays of frequencies,
and for distinct frequencies the corresponding measures
are obtained by restricting Cesàro--Bernoulli measures.

To summarize, we provided a new way of studying central
measures, which is based not on computations and estimates,
but on the intrinsic properties of these measures. Of course,
it can also be applied to discrete graphs, but, quite surprisingly, it
was realized for continuous graphs. Applications
to discrete graphs are to be considered in a separate paper.
The calculations of transition probabilities also
get a new interpretation.

For the Gelfand--Tsetlin graph
of growing rank (the graph of interlacing sequences),
we have countably many frequencies with a finite
sum of absolute values, and if  all these frequencies
are distinct, the
same restriction argument works.

This method reduces an explicit evaluation
of central measures with given
frequencies to a probabilistic problem of finding the ``Cesàro--Bernoulli'' measures
of the corresponding cones (in order to
compute the restriction).
An alternative, more standard,
way is to apply a continuous version of the Lindstrom--Gessel--Viennot lemma.

As to general Gelfand--Tsetlin
type graphs, we conjecture that for them  central measures
are also uniquely determined by
limiting frequencies.

\medskip
We are grateful to the referees for numerous useful suggestions.

\end{document}